\documentclass[a4paper,12pt]{article}

\usepackage{color}

\usepackage{makeidx}
\usepackage{bm}
\usepackage{latexsym}
\usepackage{amscd}
\usepackage{amsmath}
\usepackage{amssymb}

\usepackage{stmaryrd}

\usepackage{amsthm}
\usepackage{float}
\usepackage{graphicx}
\usepackage{psfrag}

\theoremstyle{plain}
\newtheorem{theorem}{Theorem}[section]
\newtheorem{corollary}[theorem]{Corollary}

\theoremstyle{definition}
\newtheorem{definition}[theorem]{Definition}
\newtheorem{remark}[theorem]{Remark}
\newtheorem{example}[theorem]{Example}
\newtheorem{proposition}[theorem]{Proposition}

\newcommand{\bC}{\ensuremath{\mathbb{C}}}

\newcommand{\bP}{\ensuremath{\mathbb{P}}}

\newcommand{\bR}{\ensuremath{\mathbb{R}}}

\newcommand{\scA}{\ensuremath{\mathcal{A}}}
\newcommand{\scB}{\ensuremath{\mathcal{B}}}

\newcommand{\scF}{\ensuremath{\mathcal{F}}}

\newcommand{\scL}{\ensuremath{\mathcal{L}}}

\newcommand{\sfM}{\mathsf{M}}

\newcommand{\ch}{\mathsf{ch}}

\newcommand{\Sep}{\operatorname{Sep}}
\newcommand{\Ker}{\operatorname{Ker}}

\newcommand{\A}{\scA}
\newcommand{\Coker}{\operatorname{Coker}}

\newcommand{\R}{\bR}
\newcommand{\RB}{\operatorname{RB}}

\definecolor{deepblue}{cmyk}{0,0.83,1,0.70}
\definecolor{gray}{cmyk}{0,0,0,0.3}
\definecolor{rred}{cmyk}{0,1,1,0}
\definecolor{chairo}{cmyk}{0,0.83,1,0.70}
\definecolor{roypur}{cmyk}{0.75,0.90,0,0.1}
\definecolor{darkorc}{cmyk}{0.40,0.80,0.20,0}
\definecolor{oliv}{cmyk}{0.64,0.00,0.75,0.56}
\definecolor{azuro}{cmyk}{1,1,0,0.46}

\title{Resonant bands and local system cohomology groups for 
real line arrangements}
\author{Masahiko Yoshinaga\thanks{Department of Mathematics, 
Hokkaido University, 
North 10, West 8, Kita-ku, 
Sapporo, 060-0810, 
JAPAN, 
Email: yoshinaga@math.sci.hokudai.ac.jp}
}
\date{\today}
\begin{document}
\maketitle

\begin{abstract}
We give a new algorithm computing local system 
cohomology groups for complexified real line arrangements. 
Using it, we obtain several conditions for the first local system 
cohomology to vanish and to be at most one-dimensional, which 
generalize a result by Cohen-Dimca-Orlik. 
The conditions are described in terms of discrete geometric 
structures of real figures. The proof is based on 
a recent study on minimal cell structures. 
We also compute the characteristic 
variety of the deleted $B_3$-arrangement. 


\end{abstract}

\section{Introduction}

In the theory of hyperplane arrangements, one of the central problems 
is to what extent topological invariants of the complements are 
determined combinatorially. For example, the cohomology ring 
is combinatorially determined (Orlik and Solomon), 
while the fundamental group is not 
(Rybnikov). Between these two cases, 
local system cohomology groups and monodromy eigenspaces 
of Milnor fibers recently received a considerable amount of attention. 


There are several ways to compute local system cohomology groups, 
especially for line arrangements. In this paper, we use 
the twisted minimal complex in \cite{yos-lef, yos-ch, yos-loc}. 
Since the complex is described in terms of adjacency relations of chambers, 
we can employ discrete-geometric arguments to the computation of 
local system cohomology groups. 
By using combinatorial arguments, we obtain several conditions on 
rank-one local systems for the first cohomology to vanish, also 
a condition for the first cohomology is at most one-dimensional. 

The paper is organized as follows. 
In \S\ref{sec:pre} we fix some notation and 
recall the twisted minimal complex. 
\S\ref{sec:main} is the main section of the paper. 
First, in \S\ref{sec:resban} 
we introduce discrete geometric notions, the so-called 
{\em $\scL$-resonant band} and the {\em standing wave} on this band. 
These notions were first introduced in \cite{yos-mil} in the 
study of eigenspaces of monodromy action on the Milnor fibers. 
They are also useful for the computation of local system cohomology 
groups. We give an algorithm that computes local system 
cohomology group in terms of resonant bands in 
\S\ref{sec:eigensp}. Then in \S\ref{sec:vanish} and 
\S\ref{subsec:upper}, we give a few upper bounds of 
$\dim H^1$. Finally in \S\ref{sec:ex}, we apply our 
algorithm to the deleted $B_3$-arrangement and compute 
the characteristic variety.

\section{Preliminaries}
\label{sec:pre}
\subsection{Line arrangements and local systems}
\label{sec:milnor}

Let $\A=\{H_1, \dots, H_n\}$ be an affine line arrangement in 
$\R^2$ with the defining equation $Q_\A(x, y)=\prod_{i=1}^n\alpha_i$, 
where $\alpha_i$ is a defining linear equation for $H_i$. 
In this paper, we assume that not all lines are parallel 
(or equivalently, $\A$ has at least one intersection). 
The coning $c\A$ of $\A$ is an arrangement of $n+1$ 
planes in $\R^3$ defined by the equation 
$Q_{c\A}(x, y, z)=z^{n+1}Q(\frac{x}{z}, \frac{y}{z})$. 
The line $\{z=0\}\in c\A$ is called the line at infinity and 
denoted by $H_\infty$. 
The space $\sfM(\A)=\bC^2\setminus\{Q_\A=0\}=
\bP_\bC^2\setminus\{Q_{c\A}=0\}$ is called the complexified 
complement. In this article, $\A$ always denote a line arrangement 
in $\R^2$ and $c\A$ denotes that in $\R\bP^2$. 

A rank-one local system $\scL$ is determined by a homomorphism 
$$
\pi_1(M(\A), *)\longrightarrow \bC^*. 
$$
Since the right-hand side is Abelian, the map can be lift to 
$H_1(M(\A))\longrightarrow\bC^*$. Recall that $H_1(M(\A))$ is 
a free Abelian group generated by meridian loops around 
$H_i\in\A$. Hence $\scL$ is determined by the point 
$(q_1, \dots, q_n)\in(\bC^*)^n$ of the character torus, where 
$q_i\in\bC^*$ is a monodromy along the meridian loop around 
the line $H_i$. The monodromy around the line at infinity 
$H_\infty$ is automatically determined as 
$$
q_\infty=(q_1q_2\cdots q_n)^{-1}. 
$$
Let us denote by $\scL_{\bm{q}}$ the local system determined 
by $\bm{q}=(q_1, \dots, q_n)\in\bC^n$. 
Let $X\subset\R\bP^2$ be a subset (e.g. a point or a line). 
Denote $\A_X=\{H\in c\A\mid H\supset X\}$ and $q_X=
\prod_{H\in c\A_X}q_H$. We fix some terminology. 

\begin{definition}
(1) A line $H\in c\A$ is said to be {\em resonant} if $q_H=1$. 

(2) A point $X\in\R\bP^2$ is called a {\em multiple point} if 
at least three lines in $c\A$ are passing through $X$. 

(3) A multiple point $X$ is called a {\em resonant point} if 
$q_X=1$. 
\end{definition}
We also use notations $q_{ijk}:=q_iq_jq_k$, 
$q_{ijk}^{1/2}:=q_{i}^{1/2}q_{j}^{1/2}q_{k}^{1/2}$.

\subsection{Twisted minimal cochain complexes}
\label{sec:twisted}

In this section, we recall the construction of the 
twisted minimal cochain complex from 
\cite{yos-lef, yos-ch, yos-loc}. 

A connected component of $\bR^2\setminus\bigcup_{H\in\scA}H$ 
is called a chamber. The set of all chambers is denoted by 
$\ch(\scA)$. A chamber $C\in\ch(\A)$ is called bounded 
(resp. unbounded) if the area is finite (resp. infinite). 
For an unbounded chamber $U\in\ch(\A)$, the opposite unbounded 
chamber is denoted by $U^\lor$ (see \cite[Definition 2.1]{yos-loc} for 
the definition; see also Figure \ref{fig:numbering} below). 

Let $\scF$ be a generic flag in $\bR^2$
$$
\scF:
\emptyset=
\scF^{-1}\subset
\scF^{0}\subset
\scF^{1}\subset
\scF^{2}=\bR^2, 
$$
where $\scF^k$ is a generic $k$-dimensional 
affine subspace. 

\begin{definition}
For $k=0, 1, 2$, define the subset 
$\ch^k_\scF(\scA)\subset\ch(\scA)$ by 
$$
\ch^k_\scF(\scA):=
\{C\in\ch(\scA)\mid C\cap\scF^k\neq\emptyset, 
C\cap\scF^{k-1}=\emptyset\}.
$$ 
\end{definition}
The set of chambers decomposes into a 
disjoint union as 
$\ch(\scA)=
\ch^0_\scF(\scA)\sqcup
\ch^1_\scF(\scA)\sqcup
\ch^2_\scF(\scA)$. The cardinality of 
$\ch^k_\scF(\scA)$ is equal to $b_k(\sfM(\A))$ 
for $k=0, 1, 2$.

We further assume that 
the generic flag $\scF$ satisfies the following 
conditions: 
\begin{itemize}
\item $\scF^1$ does not separate intersections of $\scA$, 
\item $\scF^0$ does not separate $n$-points 
$\scA\cap\scF^1$. 
\end{itemize}
Then we can choose coordinates $x_1, x_2$ so that 
$\scF^0$ is the origin $(0,0)$, 
$\scF^1$ is given by $x_2=0$, all intersections of $\scA$ 
are contained in the upper-half plane $\{(x_1, x_2)\in\bR^2\mid
x_2>0\}$ and $\scA\cap\scF^1$ is contained in the 
half-line $\{(x_1, 0)\mid x_1>0\}$. 

We set $H_i\cap\scF^1$ to have coordinates $(a_i, 0)$. 
By changing the numbering of the lines and the signs of 
the defining equation $\alpha_i$ of $H_i\in\scA$ 
we may assume that 
\begin{itemize}
\item $0<a_1<a_2<\dots<a_n$, 
\item the origin $\scF^0$ is contained in the negative 
half-plane $H_i^-=\{\alpha_i<0\}$. 
\end{itemize}
We set 
$\ch_0^\scF(\scA)=\{U_0\}$ and 
$\ch_1^\scF(\scA)=\{U_1, \dots, U_{n-1}, U_0^\lor\}$ so that 
$U_p\cap\scF^1$ is equal to the interval $(a_p, a_{p+1})$ for 
$p=1, \dots, n-1$. 
It is easily seen that the 
chambers $U_0, U_1, \dots, U_{n-1}$ and $U_0^\lor$ 
have the following expression: 
\begin{equation}
\begin{split}
&U_0=\bigcap_{i=1}^n\{\alpha_i<0\},\\
&U_p=
\bigcap_{i=1}^{p}\{\alpha_i>0\}\cap
\bigcap_{i=p+1}^n\{\alpha_i<0\},\ (p=1, \dots, n-1),\\
&U_0^\lor=
\bigcap_{i=1}^{n}\{\alpha_i>0\}.
\end{split}
\end{equation}
The notations introduced to this point are illustrated 
in Figure \ref{fig:numbering}.

\begin{figure}[htbp]
\begin{picture}(100,100)(20,0)
\thicklines

\put(70,20){\circle*{4}}
\put(40,27){$\scF^0(0,0)$}

\multiput(50,20)(3,0){83}{\circle*{1}}
\put(300,20){\vector(1,0){0}}
\put(285,24){$\scF^1$}

\put(280,0){\line(-2,1){200}}
\put(240,20){\circle*{3}}
\put(243,24){$a_5$}
\put(278,-10){$H_5$}

\put(200,0){\line(0,1){100}}
\put(200,20){\circle*{3}}
\put(203,24){$a_4$}
\put(195,-10){$H_4$}

\put(170,0){\line(0,1){100}}
\put(170,20){\circle*{3}}
\put(173,24){$a_3$}
\put(165,-10){$H_3$}

\put(140,0){\line(0,1){100}}
\put(140,20){\circle*{3}}
\put(143,24){$a_2$}
\put(135,-10){$H_2$}

\put(60,0){\line(2,1){200}}
\put(100,20){\circle*{3}}
\put(94,24){$a_1$}
\put(55,-10){$H_1$}

\put(118,52){$U_0$}
\put(210,52){$U_0^\lor$}

\put(118,0){$U_1$}
\put(150,0){$U_2$}
\put(180,0){$U_3$}
\put(210,0){$U_4$}

\put(118,90){$U_4^\lor$}
\put(150,90){$U_2^\lor$}
\put(180,90){$U_3^\lor$}
\put(210,90){$U_1^\lor$}

\put(143,52){$C_1$}
\put(183,52){$C_2$}

\put(280,92){$\ch^0_\scF(\scA)=\{U_0\}$}
\put(280,75){$\ch^1_\scF(\scA)=\{U_0^\lor, U_1,\dots,U_4\}$}
\put(280,58){$\ch^2_\scF(\scA)=\{U_1^\lor,\dots,U_4^\lor,
C_1,C_2\}$}

\end{picture}
     \caption{Numbering of lines and chambers.}\label{fig:numbering}
\end{figure}
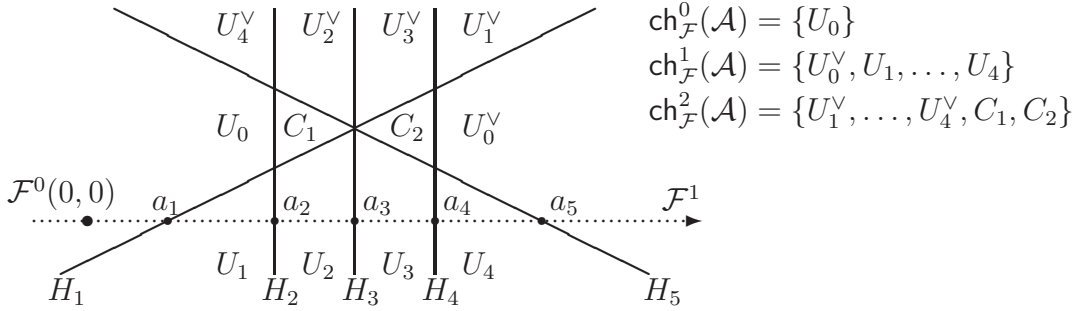

Let $\scL$ be a complex rank-one local system 
on $\sfM(\scA)$. The local system $\scL$ is determined by 
non-zero complex numbers (monodromy around $H_i$) 
$q_i\in\bC^*$, $i=1, \dots, n$. 
Fix a square root $q_i^{1/2}\in\bC^*$ for each $i$. 

\begin{definition}
\label{def:sep}
$(1)$ 
For $C, C'\in\ch(\A)$, let us denote by 
$\Sep(C,C')$ the set of lines $H_i\in\A$ which 
separate $C$ and $C'$. \\
$(2)$ 
Define the complex number $\Delta(C,C')\in\bC$ by 
$$
\Delta(C,C'):=
\prod_{H_i\in\Sep(C,C')}q_i^{1/2}
-
\prod_{H_i\in\Sep(C,C')}q_i^{-1/2}. 
$$
\end{definition}
Now we construct the cochain complex $(\bC[\ch^\bullet_\scF(\scA)], d_\scL)$. 
\begin{itemize}
\item[(i)] The map 
$d_\scL:\bC[\ch^0_\scF(\scA)]
\longrightarrow\bC[\ch^1_\scF(\scA)]$ is defined by 
$$
d_\scL([U_0])=
\Delta(U_0,U_0^\lor)[U_0^\lor]+\sum_{p=1}^{n-1}\Delta(U_0, U_p)[U_p]. 
$$
\item[(ii)] 
$d_\scL:\bC[\ch^1_\scF(\scA)]
\longrightarrow\bC[\ch^2_\scF(\scA)]$ is defined by 
\begin{equation*}
\begin{split}
d_\scL([U_p])&=
-\sum_{\substack{C\in\ch^2_\scF(\A) \\ \alpha_p(C)>0\\ \alpha_{p+1}(C)<0}}\Delta(U_p, C)[C]
+\sum_{\substack{C\in\ch^2_\scF(\A) \\ \alpha_p(C)<0\\ \alpha_{p+1}(C)>0}}\Delta(U_p, C)[C],\ (\mbox{for }p=1, \dots, n-1),
\\
d_\scL([U_0^\lor])&=
-\sum_{\alpha_{n}(C)>0}\Delta(U_0^\lor, C)[C]. 
\end{split}
\end{equation*}
\end{itemize}

\begin{example}
Let $\A=\{H_1, \dots, H_5\}$, and let the flag $\scF$ be as in Figure 
\ref{fig:numbering}. Then 
$$
d_\scL([U_0])=
([U_1], [U_2], [U_3], [U_4], [U_0^\lor])
\begin{pmatrix}
q_{1}^{1/2}-q_{1}^{-1/2}\\
q_{12}^{1/2}-q_{12}^{-1/2}\\
q_{123}^{1/2}-q_{123}^{-1/2}\\
q_{1234}^{1/2}-q_{1234}^{-1/2}\\
q_{12345}^{1/2}-q_{12345}^{-1/2}
\end{pmatrix}, 
$$
\begin{equation*}
\begin{split}
&d_\scL([U_1], [U_2], [U_3], [U_4], [U_0^\lor])
=
([U_1^\lor], [U_2^\lor], [U_3^\lor], [U_4^\lor], [C_1], [C_2])\\
&\times
\begin{pmatrix}
q_{12345}^{1/2}-q_{12345}^{-1/2}
&0
&0
&0
&-(q_{1}^{1/2}-q_{1}^{-1/2})\\ 
q_{125}^{1/2}-q_{125}^{-1/2}
&-(q_{15}^{1/2}-q_{15}^{-1/2})
&0
&q_{1345}^{1/2}-q_{1345}^{-1/2}
&-(q_{134}^{1/2}-q_{134}^{-1/2})\\ 
q_{1235}^{1/2}-q_{1235}^{-1/2}
&0
&-(q_{15}^{1/2}-q_{15}^{-1/2})
&q_{145}^{1/2}-q_{145}^{-1/2}
&-(q_{14}^{1/2}-q_{14}^{-1/2})\\ 
0
&0
&0
&q_{12345}^{1/2}-q_{12345}^{-1/2}
&-(q_{1234}^{1/2}-q_{1234}^{-1/2})\\ 
q_{12}^{1/2}-q_{12}^{-1/2}
&-(q_{1}^{1/2}-q_{1}^{-1/2})
&0
&0
&0\\ 
0
&0
&-(q_{5}^{1/2}-q_{5}^{-1/2})
&q_{45}^{1/2}-q_{45}^{-1/2}
&-(q_{4}^{1/2}-q_{4}^{-1/2}) 
\end{pmatrix}.
\end{split}
\end{equation*}
\end{example}

\begin{theorem}
\label{thm:twist}
Under the above notation, 
$(\bC[\ch^\bullet_\scF(\scA)], d_\scL)$ is a cochain complex and 
$$
H^k(\bC[\ch^\bullet_\scF(\scA)], d_\scL)
\simeq
H^k(M(\scA), \scL). 
$$
\end{theorem}
See \cite{yos-lef, yos-ch, yos-loc} for details.

\section{Main result}
\label{sec:main}

Let $\A=\{H_1, \dots, H_n\}$ be an arrangement of affine lines 
in $\R^2$. Let $\scF$ be a generic flag as in \S\ref{sec:twisted}. 
We will see that if $q_\infty\neq 1$, then we obtain a simpler algorithm 
computing $H^1(M(\A), \scL)$. 

\subsection{Resonant bands and standing waves}
\label{sec:resban}

\begin{definition}
\label{def:band}
A {\em band} $B$ is a region bounded by a pair of consecutive 
parallel lines $H_i$ and $H_{i+1}$. 
\end{definition}
Each band $B$ includes two unbounded chambers 
$U_1(B), U_2(B)\in\ch(\A)$. By definition, 
$U_1(B)$ and $U_2(B)$ are opposite each other, 
$U_1(B)^\lor=U_2(B)$ and $U_2(B)^\lor=U_1(B)$. 

\begin{definition}
A band $B$ is called {\em $\scL$-resonant} if 
$$
\Delta(U_1(B), U_2(B))=0. 
$$
We denote the set of all $\scL$-resonant bands 
by $\RB_\scL(\A)$. 
\end{definition}

Let us denote by $\overline{B}$ the closure of $B$ in 
the real projective plane $\R\bP^2$. Then the intersection 
$X(B):=\overline{B}\cap H_\infty$ is a single point. 
Each line $H\in\A\cup\{H_\infty\}$ 
either passes through $X(B)$ or 
separates $U_1(B)$ and $U_2(B)$. 

\begin{proposition}
\label{prop:resban}
A band $B$ is $\scL$-resonant if and only if 
$q_{X(B)}=1$. The directions of $\scL$-resonant 
bands are one-to-one correspondence with points 
$X\in H_\infty$ such that $q_X=1$. 
\end{proposition}
\begin{proof}
By the above remark, $c\A$ is decomposed as 
$$
c\A=\Sep(U_1(B), U_2(B))\sqcup (c\A)_{X(B)}, 
$$
and we have 
$$
q_{X(B)}\times\prod_{H_i\in\Sep(U_1, U_2)}q_i=
\prod_{H_i\in c\A}q_i=1. 
$$
Hence $q_{X(B)}=1$ if and only if 
$\prod_{H_i\in\Sep(U_1, U_2)}q_i=1$, which is equivalent 
to $\Delta(U_1(B), U_2(B))=0$. 
\end{proof}


To an $\scL$-resonant band $B\in\RB_\scL(\A)$ we can associate 
a {\em standing wave} 
$\nabla_i(B)\in\bC[\ch(\A)]$ $(i=1,2)$ on the band $B$ 
as follows: 
\begin{equation}
\label{eq:wave}
\nabla_i(B)=
\sum_{\substack{C\in\ch(\A),\\ C\subset B}}
\Delta(U_i(B), C)\cdot[C]. 
\end{equation}

\begin{proposition}
Let $B$ be a $\scL$-resonant band. Then 
$\nabla_2(B)=\pm\nabla_1(B)$. 
\end{proposition}
\begin{proof}
Since $B$ is $\scL$-resonant, 
$$
\varepsilon:=
q_{X(B)}^{1/2}=
\prod_{H_i\in(c\A)_{X(B)}}q_i^{1/2}
$$
is either $+1$ or $-1$. 
Let $C\subset B$ be a chamber contained in $B$. 
Then $c\A$ is decomposed as 
$
c\A=
\Sep(U_1(B), C)\sqcup
\Sep(U_2(B), C)\sqcup
(c\A)_{X(B)}, 
$
and so we have 
$$
\prod_{H_i\in\Sep(U_1,C)}q_i^{1/2}
\times
\prod_{H_i\in\Sep(U_2,C)}q_i^{1/2}
=\varepsilon. 
$$
Hence 
$$
\Delta(U_1(B),C)=-\varepsilon\cdot
\Delta(U_2(B),C). 
$$
We conclude that 
$\nabla_1(B)=-\varepsilon\cdot\nabla_2(B)$. 
\end{proof}
In the remainder of the paper 
we denote $\nabla(B):=\nabla_1(B)$ for simplicity. 


\begin{remark}
To indicate the choice of $U_1(B)$ and $U_2(B)$, we always 
put the name $B$ of the band in the unbounded chamber 
$U_1(B)$ in figures. 
\end{remark}

\subsection{Cohomology via resonant bands}
\label{sec:eigensp}

The map $B\longmapsto\nabla(B)$ can be naturally extended to the linear map 
\begin{equation}
\label{eq:nabla}
\nabla:
\bC[\RB_\scL(\A)]\longrightarrow
\bC[\ch(\A)]. 
\end{equation}

\begin{theorem}
\label{thm:main}
Assume that $q_\infty\neq 1$. Then 
the kernel of $\nabla$ is isomorphic to 
the first cohomology of $\scL$-coefficients, that is, 
$$
\Ker\left(\nabla:
\bC[\RB_\scL(\A)]\longrightarrow
\bC[\ch(\A)]\right)\simeq 
H^1(M(\A), \scL).
$$
In particular, $\dim H^1(M(\A), \scL)$ is equal to the number of 
linear relations among the standing waves 
$\nabla(B)$, $B\in\RB_\scL(\A)$. 
\end{theorem}

\begin{proof}

We consider the first cohomology group 
$H^1(\bC[\ch^\bullet_\scF(\A)], d_\scL)$ of the twisted minimal 
cochain complex. The image $d_\scL:\bC[\ch^0_\scF(\A)]
\longrightarrow \bC[\ch^1_\scF(\A)]$ is generated by 
$$
d_\scL([U_0])=
\sum_{p=1}^{n-1}
\Delta(U_0, U_p)
[U_p] 
\pm
(q_\infty^{1/2}-q_\infty^{-1/2})
[U_0^\lor]. 
$$
(The sign depends on $q_{12\dots n\infty}^{1/2}=\pm 1$.) 
Since $q_\infty\neq 1$, 
the coefficient of $[U_0^\lor]$ in $d_\scL([U_0])$ is non-zero. 
Define the subspace $V$ of $\bC[\ch^1_\scF(\A)]$ by 
\begin{equation}
\begin{split}
V&=\bigoplus_{p=1}^{n-1}\bC\cdot[U_p]\\
(&\simeq \Coker
\left(d_\scL:\bC[\ch^0_\scF(\A)]
\longrightarrow \bC[\ch^1_\scF(\A)]\right)). 
\end{split}
\end{equation}
Then $H^1(\bC[\ch^\bullet_\scF(\A)], d_\scL)$ is isomorphic to 
$\Ker\left(d_\scL|_V:V\longrightarrow\bC[\ch^2_\scF(\A)]\right)$. 
It is sufficient to show that 
$\Ker (d_\scL|_V)\simeq \Ker\nabla$, which will be done in several steps. 
Suppose that 
$\varphi=\sum_{p=1}^{n-1}c_p\cdot[U_p]\in \Ker(d_\scL|_V)$. 

\begin{itemize}
\item[(i)] 
If $H_i$ and $H_{i+1}$ are not parallel, then $c_i=0$. 
\end{itemize}
Note that if $j\neq i$, then the chamber $[U_i^\lor]$ does not appear 
in $d_\scL([U_{j}])$. Thus the coefficient of $[U_i^\lor]$ in 
$$
d_\scL(\varphi)=
\sum_{p=1}^{n-1}c_p\cdot d_\scL([U_p])
$$
is $c_i\cdot\Delta(U_i, U_i^\lor)=
\pm c_i(q_\infty^{1/2}-q_\infty^{-1/2})$. This equals zero if and only if 
$c_i=0$. 

Now we may assume that 
$\varphi=\sum_{p}c_p\cdot[U_p]\in\Ker(d_\scL)$ is a linear combination 
of $[U_p]$s such that $H_p$ and $H_{p+1}$ are parallel. 
Suppose that $H_i$ and $H_{i+1}$ are parallel and denote by  
$B_i$ the band determined by these lines. 
\begin{itemize}
\item[(ii)] 
If $B_i$ is not $\scL$-resonant, then $c_i=0$. 
\end{itemize}
In this case, $\Delta(U_i, U_i^\lor)\neq 0$. 
Since 
$\varphi$ is a linear combination of $[U_p]$s with 
parallel boundaries $H_p$ and $H_{p+1}$, the term $[U_i^\lor]$ 
appears only in $d_\scL([U_i])$, which is equal to 
$c_i\cdot\Delta(U_i, U_i^\lor)[U_i^\lor]$. Therefore $c_i=0$. 

Finally we may assume that $\varphi$ is a linear combination of 
$[U_p]$s such that the boundaries $H_p$ and $H_{p+1}$ are parallel and 
the length of the corresponding band $B_p$ is divisible by $k$. In this case, 
it is straightforward to check that the maps $d_\scL$ and $\nabla$ are 
identical. This completes the proof. 
\end{proof}

\begin{example}
Let $\A=\{H_1, H_2, H_3, H_4, H_5\}$ be the line arrangement in 
Figure \ref{fig:numbering}. There are two bands: 
\begin{equation*}
\begin{split}
B_1=&\mbox{(The band bounded by $H_2$ and $H_3$),}\\
B_2=&\mbox{(The band bounded by $H_3$ and $H_4$).}
\end{split}
\end{equation*}
We set $U_1(B_1)=U_2, 
U_2(B_1)=U_2^\lor, 
U_1(B_2)=U_3$ and $
U_2(B_2)=U_3^\lor$. 
The band $B_i$ is $\scL_{\bm{q}}$-resonant if and only if 
$q_{15}=1$, or equivalently $q_{234\infty}=1$. 
Then we have 
\begin{equation*}
\begin{split}
\nabla(B_1)=&\Delta(U_2, C_1)\cdot [C_1]=(q_1^{1/2}-q_1^{-1/2}) \cdot [C_1]\\
\nabla(B_2)=&\Delta(U_3, C_2)\cdot [C_2]=(q_5^{1/2}-q_5^{-1/2}) \cdot [C_2]
=\pm (q_1^{1/2}-q_1^{-1/2}) \cdot [C_2].
\end{split}
\end{equation*}
Hence, 
\begin{equation*}
\dim H^1(\scL_{\bm{q}})=
\left\{
\begin{array}{cc}
2&\ q_1=q_5=1\\
0&\ \mbox{otherwise}. 
\end{array}
\right.
\end{equation*}
This example is a special case of Theorem \ref{thm:1pt}. 
\end{example}

\subsection{Vanishing}
\label{sec:vanish}

We describe some corollaries to Theorem \ref{thm:main}. 

\begin{corollary}
\label{cor:empty}
Suppose that $q_\infty\neq 1$. 
If every multiple point $X\in H_\infty$ satisfies 
$q_X\neq 1$, then $H^1(M(\A),\scL)=0$. 
\end{corollary}
\begin{proof}
By Proposition \ref{prop:resban}, the assumption is equivalent to 
$\RB_\scL(\A)=\emptyset$. 
Since $\bC[\RB_k(\A)]=0$, obviously 
$\Ker(\nabla:\bC[\RB_k(\A)]\rightarrow\bC[\ch(\A)])=0$. 
By Theorem \ref{thm:main}, $H^1(M(\A),\scL)=0$. 
\end{proof}

\begin{corollary}
\label{cor:cdo}
If there exists a non-resonant line $H_i\in c\A$ such that 
every multiple point on $H_i$ is non-resonant, 
then $H^1(M(\A),\scL)=0$. 
\end{corollary}

\begin{remark}
Corollary \ref{cor:cdo} 
is proved by Cohen, Dimca and Orlik \cite{cdo} for 
more general complex arrangement cases. 
It is also proved for real line arrangements that the assumption of 
Corollary \ref{cor:cdo} is 
equivalent to $H^2(M(\A), \scL)$ being generated 
by bounded chambers \cite{yos-loc}. 
\end{remark}

For real case, we obtain a modified version. 

\begin{theorem}
\label{thm:1pt}
Suppose that $q_\infty\neq 1$ and $H_\infty$ has a unique 
resonant multiple point 
$X$. Let 
$\scB=c\A\setminus(c\A)_X$ be the set of lines which are not 
passing through $X$. Then 
\begin{equation}
\dim H^1(M(\A), \scL)=
\left\{
\begin{array}{cl}
|(c\A)_X|-2,&\mbox{if }q_H=1, \forall H\in\scB\\
0&\mbox{otherwise.}
\end{array}
\right.
\end{equation}
\end{theorem}
\begin{proof}
By the assumption, 
$\RB_\scL(\A)=\{B_1, \dots, B_m\}$ consists of parallel bands, 
where $m=|(c\A)_X|-2$. 
Now the supports of $\nabla(B_1), \dots, \nabla(B_m)$, that is, 
the set of chambers appearing in each standing wave, 
are mutually disjoint. Therefore, 
$\nabla(B_1), \dots, \nabla(B_m)$ are linearly dependent 
if and only if 
$\nabla(B_1)=\dots=\nabla(B_m)=0$, which is equivalent to 
that $q_H=1$ for all $H\in\scB$. In this case, we have that 
the space $\Ker(\nabla)=\bC[\RB_\scL(\A)]$ is 
$(|(c\A)_X|-2)$-dimensional. 
\end{proof}

\begin{corollary}
\label{cor:1pt}
Suppose that there exists a line $H_i\in c\A$ such that 
$q_i\neq 1$ and there is at most one point $X\in H_i$ satisfying 
$q_X=1$. Then $\dim H^1(M(\A), \scL)$ is combinatorially determined. 
\end{corollary}

\begin{remark}
We do not know whether Theorem \ref{thm:1pt} holds for 
complex arrangements. 
\end{remark}

\subsection{Upper-bound}
\label{subsec:upper}

Recall that two lines $H, H'$ in the real projective plane $\R\bP^2$ 
divide the space into two regions. 

\begin{definition}
\label{def:sharp}
Let $c\A$ be a line arrangement in the real projective plane $\R\bP^2$. 
A pair of non-resonant lines $H_i$ and $H_j\in c\A$ 
(that is, they satisfy $q_i\neq 1$ and $q_j\neq 1$) 
is said to be a {\em sharp pair} if 
all intersection points of $c\A\setminus 
\{H_i, H_j\}$ are contained in one of two regions or lie 
on $H_i\cup H_j$. (In other words, there are no intersection points 
in one of the two regions determined by $H_i$ and $H_j$.) 
\end{definition}

\begin{example}
In Figure \ref{fig:A(12,2)}, the lines $H_1$ and $H_\infty$ form 
a sharp pair (since the left half plane of $H_1$ does not contain 
intersections). 
\end{example}


If there exists a non-resonant line $H\in c\A$ such that $H$ has at most one 
resonant multiple point, then $\dim H^1(M(\A), \scL)$ is computed 
combinatorially (Corollary \ref{cor:1pt}). 
Thus we may assume that every non-resonant line $H\in c\A$ has 
at least two resonant multiple points. 

\begin{theorem}
\label{thm:sharp}
Let $\A$ be a line arrangement and $\scL$ be a rank-one local system. 
Assume that every non-resonant line $H\in c\A$ has at least 
two resonant points. Suppose that 
the arrangement $c\A$ contains a sharp pair of non-resonant lines. 
Then: 
\begin{itemize}
\item[(i)] 
$\dim H^1(M(\A), \scL)\leq 1$. 
\item[(ii)] 
Suppose that the lines $H_1, H_2\in c\A$ are non-resonant 
and form a sharp pair. Let $X=H_1\cap H_2$ 
be the intersection. If $(c\A)_X=\{H_1, H_2\}$ or $q_X\neq 1$, 
then $H^1(M(\A), \scL)=0$. 
\end{itemize}
\end{theorem}
\begin{proof}
By $PGL_3(\bC)$ action, we may assume that 
the line at infinity $H_\infty$ and $H_1=\{x=0\}$ 
form a sharp pair of non-resonant lines and there is no intersections in 
the region $\{(x,y)\in\R^2\mid x<0\}$ (see Figure \ref{fig:A(12,2)}). 
The intersection is $X=H_\infty\cap H_1=\{(0:1:0)\}$. Let $B$ be 
a horizontal (that is, non-vertical) band, 
that is a band which is not passing through 
the point $X$. 
We choose $U_1(B)$ to be the unbounded chamber in $B$ contained 
in the region $\{x<0\}$. 
Denote by $C_B$ the leftmost bounded chamber 
in $B$ (e.g., in Figure \ref{fig:A(12,2)}, there are four 
horizontal band $B_1, \dots, B_4$. In this case, 
$C_{B_1}=C_2, 
C_{B_2}=C_3, C_{B_3}=C_4$ and $C_{B_4}=C_6$). 

\begin{figure}[htbp]
\begin{picture}(400,150)(20,0)
\thicklines

\multiput(140,0)(40,0){4}{\line(0,1){150}}
\put(135,-8){\footnotesize $H_1$}
\multiput(100,45)(0,30){3}{\line(1,0){200}}
\put(100,15){\line(4,3){180}}
\put(120,0){\line(4,3){180}}
\put(100,135){\line(4,-3){180}}
\put(120,150){\line(4,-3){180}}

\put(111,133){$B_1$}
\put(111,85){$B_2$}
\put(111,55){$B_3$}
\put(111,7){$B_4$}
\put(155,0){$B_5$}
\put(195,0){$B_6$}
\put(235,0){$B_7$}

\put(148,112){\small $C_1$}
\put(164,92){\small $C_2$}
\put(148,82){\small $C_3$}
\put(148,63){\small $C_4$}
\put(164,50){\small $C_5$}
\put(148,33){\small $C_6$}

\end{picture}
      \caption{$H_1$ and $H_\infty$ form a sharp pair}
\label{fig:A(12,2)}
\end{figure}
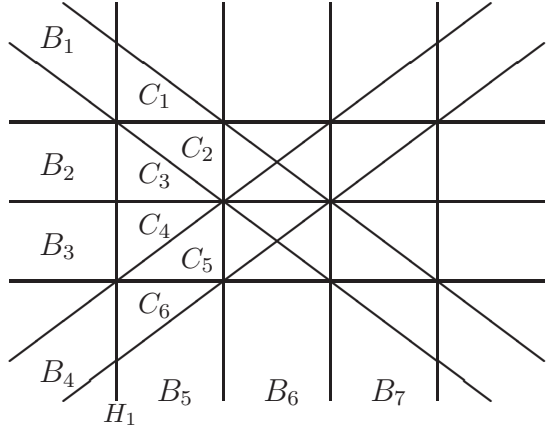

First consider the case $q_X\neq 1$. 
Then all $\scL$-resonant bands are horizontal. 
Let $B\in\RB_\scL(\A)$. Then 
\begin{equation}
\label{eq:nozero}
\nabla(B)=
\Delta(U_1(B), C_B)
\cdot[C_B]+\cdots. 
\end{equation}
Thus $[C_B]$ has a non-zero coefficient. 
Note that $C_B$ is contained in the unique $\scL$-resonant band $B$. 
Hence the standing waves 
$\nabla(B), B\in\RB_k(\A)$ are linearly independent. 
Thus (ii) is proved. 

Now we assume that $q_X=1$. In this case, there are vertical 
$\scL$-resonant bands. Denote by $B_{left}$ the 
leftmost vertical band 
(in Figure \ref{fig:A(12,2)}, $B_{left}=B_5$). Suppose that 
$$
c_{left}\cdot B_{left}+\cdots \in\Ker(\nabla). 
$$
Let $B\in\RB_\scL(\A)$ be a horizontal $\scL$-resonant band. 
Then, since $C_B$ is contained in only $B$ and $B_{left}$, 
the coefficient $c_{left}$ of $B_{left}$ determines 
the coefficient of $B$. The coefficients of other 
vertical $\scL$-resonant bands are also determined by 
those of the horizontal bands. Hence $\Ker(\nabla)$ is 
at most one-dimensional. 
\end{proof}

\begin{corollary}
\label{cor:interior}
Suppose that $\dim H^1(M(\A), \scL)\geq 2$. Then for every sharp pair 
$H, H'\in c\A$ of non-resonant lines, there exist intersections of 
$c\A$ in both two regions determined by $H$ and $H'$. 
\end{corollary}

\begin{example}
\label{ex:2dim}
Let $\A=\{H_1, \dots, H_8\}$ 
be the deleted $B_3$-arrangement (see \S\ref{subsec:delb3} for 
the definition) and let 
$\bm{q}_+=(q_1, \dots, q_8)=(1,-1,-1,1,1,-1,1,-1)$. We compute 
$H^1(M(\A), \scL_{\bm{q}_+})$. 
Since $q_8\neq 1$, we can apply Theorem \ref{thm:main} to 
Figure \ref{fig:h8}. 
There are four resonant bands: $\RB_{\scL_{\bm{q}_+}}=
\{B_1, B_2, B_3, B_4\}$. 
Set $q^{1/2}_1=q^{1/2}_4=q^{1/2}_5=q^{1/2}_7=1$ and 
$q^{1/2}_2=q^{1/2}_3=q^{1/2}_6=q^{1/2}_8=i$. 
Then we have 
$$
\nabla(B_1)=\nabla(B_2)=\nabla(B_4)=2i\cdot C_2,\ 
\nabla(B_3)=2i(C_4+C_5+C_6), 
$$
and so $B_1-B_2, B_2-B_3\in\Ker(\nabla)$ form a basis. We have 
$\dim H^1(M(\A), \scL_{\bm{q}_+})=2$. 

If we set 
$\bm{q}_-=(q_1, \dots, q_8)=(-1,1,1,-1,1,-1,1,-1)$, 
similarly we have 
$$
\nabla(B_1)=\nabla(B_3)=\nabla(B_4)=2i\cdot C_5,\ 
\nabla(B_2)=2i(C_1+C_2+C_3), 
$$
and $\dim H^1(M(\A), \scL_{\bm{q}_-})=2$. 
\end{example}

\section{Example: The deleted $B_3$-arrangement} 
\label{sec:ex}


Using Theorem \ref{thm:main}, we can compute the characteristic variety 
of line arrangements. 
We apply the following strategy to the deleted $B_3$-arrangement. 
\begin{itemize}
\item[(1)] Fix a line $H$. 
\begin{itemize}
\item Put $H$ at $\infty$ and assume that $q_H\neq 1$. 
\item Using Theorem \ref{thm:main}, compute the characteristic 
variety using the assumption that $q_H\neq 1$. 
\end{itemize}
\item[(2)] Assume that $q_H=1$ and choose another line $H'$, go back to (1). 
\end{itemize}

\subsection{Deleted $B_3$-arrangement} 
\label{subsec:delb3}

The deleted $B_3$-arrangement is an arrangement of $8$ lines 
in $\R\bP^2$ defined by the following equations. 
\begin{equation}
\begin{split}
H_1:\ & y=0\\
H_2:\ & y-z=0\\
H_3:\ & x=0\\
H_4:\ & x-z=0\\
H_5:\ & x-y+z=0\\
H_6:\ & x-y=0\\
H_7:\ & x-y-z=0\\
H_8:\ & z=0. 
\end{split}
\end{equation}
(We use the numbering in \cite[Example 10.6]{suc-fundam}. 
See 
Figure \ref{fig:h5}, 
\ref{fig:h8}, 
\ref{fig:h6} and 
\ref{fig:h3} below.) 

In the sequel we compute the characteristic variety 
\begin{equation}
V_1(M(\A))=
\{\bm{q}=(q_1, \dots, q_8)\in(\bC^*)^8\mid
q_1q_2\cdots q_8=1,\ 
\dim H^1(M(\A), \scL_{\bm{q}})\geq 1\}
\end{equation}
of the deleted $B_3$-arrangement (without using computer). 

\subsection{The case $q_5\neq 1$}

First, we consider the case $q_5\neq 1$. There are two 
multiple points $235$ and $5678$ on $H_5$ (Figure \ref{fig:h5}). 

\begin{figure}[htbp]
\begin{picture}(350,170)(0,0)
\thicklines

\multiput(150,10)(50,0){3}{\line(0,1){150}}
\multiput(120,35)(0,100){2}{\line(1,0){160}}

\put(125,10){\line(1,1){150}}
\put(275,10){\line(-1,1){150}}

\put(142,162){$H_8$}
\put(192,162){$H_7$}
\put(242,162){$H_6$}
\put(107,162){$H_4$}
\put(107,132){$H_3$}
\put(107,32){$H_2$}
\put(110,7){$H_1$}

\put(110,85){$B_1$}
\put(170,150){$B_2$}
\put(220,150){$B_3$}

\put(164,85){$C_2$}
\put(180,114){$C_1$}
\put(180,47){$C_3$}

\put(225,85){$C_5$}

\put(209,114){$C_4$}
\put(209,47){$C_6$}

\end{picture}
      \caption{The deleted $B_3$-arrangement with line at infinity $H_5$}
\label{fig:h5}
\end{figure}
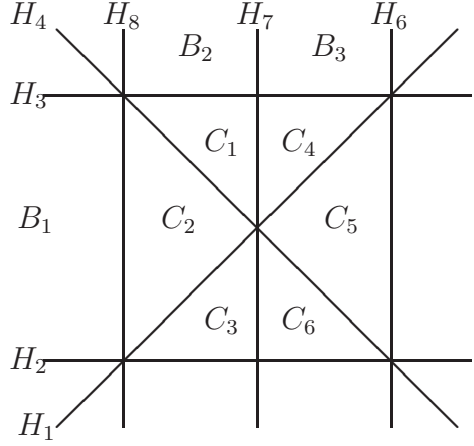
There four cases. 
\begin{itemize}
\item[(1)] $q_{235}\neq 1, q_{5678}\neq 1$. 
Then by Corollary \ref{cor:empty}, we have 
$H^1(\scL_{\bm{q}})=0$. 
\item[(2)] $q_{235}\neq 1, q_{5678}=1$. 
Then by Theorem \ref{thm:1pt}, 
$H^1(\scL_{\bm{q}})=0$ if and only if 
$q_1=q_2=q_3=q_4=1$ (and then $\dim H^1=2$). 
It is corresponding to the component $C_{5678}$ in 
\cite[Example 10.6]{suc-fundam}. 
\item[(3)] $q_{235}=1, q_{5678}\neq 1$. 
Then by Theorem \ref{thm:1pt}, 
$H^1(\scL_{\bm{q}})=0$ if and only if 
$q_1=q_4=q_6=q_7=q_8=1$ (and then $\dim H^1=1$). 
It is corresponding to the component $C_{235}$ in 
\cite[Example 10.6]{suc-fundam}. 
\item[(4)] $q_{235}=q_{5678}=1$. 
This case requires more detailed analysis as follows. 
\end{itemize}
In the case (4), there are three resonant bands 
$\RB_{\scL_{\bm{q}}}=\{B_1, B_2, B_3\}$. 
The standing waves $\nabla(B_i)$s are computed as follows. 
\begin{equation*}
\begin{pmatrix}
\nabla(B_1)\\
\nabla(B_2)\\
\nabla(B_3)
\end{pmatrix}
=
\begin{pmatrix}
\Delta(48)&\Delta(8)&\Delta(18)&\Delta(478)&\Delta(1478)&\Delta(178)\\
\Delta(3)&\Delta(34)&\Delta(134)&0&0&0\\
0&0&0&\Delta(3)&\Delta(13)&\Delta(134)
\end{pmatrix}
\begin{pmatrix}
C_1\\
C_2\\
C_3\\
C_4\\
C_5\\
C_6
\end{pmatrix}, 
\end{equation*}
where we use the abbreviation 
$\Delta(ijk):=q^{1/2}_{ijk}-q^{-1/2}_{ijk}$. 
It is easily seen that $\Ker\nabla\neq 0$ if and only if the following 
six $2\times 2$ minors are all zero. 
\begin{equation*}
\begin{split}
D_1=
\det
\begin{pmatrix}
\Delta(48)&\Delta(8)\\
\Delta(3)&\Delta(34)
\end{pmatrix}
=&\Delta(4)\Delta(348),\\
D_2=
\det
\begin{pmatrix}
\Delta(8)&\Delta(18)\\
\Delta(34)&\Delta(134)
\end{pmatrix}
=&\pm \Delta(1)\Delta(128),\\
D_3=
\det
\begin{pmatrix}
\Delta(48)&\Delta(18)\\
\Delta(3)&\Delta(134)
\end{pmatrix}
=&\Delta(4)\Delta(1348)+q^{1/2}_{1234}\Delta(1)\Delta(1248),\\
D_4=
\det
\begin{pmatrix}
\Delta(478)&\Delta(1478)\\
\Delta(3)&\Delta(13)
\end{pmatrix}
=&\pm\Delta(1)\Delta(136),\\
D_5=
\det
\begin{pmatrix}
\Delta(1478)&\Delta(178)\\
\Delta(13)&\Delta(134)
\end{pmatrix}
=&\pm\Delta(4)\Delta(246),\\
D_6=
\det
\begin{pmatrix}
\Delta(478)&\Delta(178)\\
\Delta(3)&\Delta(134)
\end{pmatrix}
=&q^{1/2}_{1234}\Delta(4)\Delta(1246)+\Delta(1)\Delta(1346).
\end{split}
\end{equation*}
Since $5678$ and $235$ are resonant points, 
$q_{1234}=q_{14678}=1$. 
By the relations $D_1=D_2=D_4=D_5=0$, it is natural 
to divide into four cases: 
\begin{itemize}
\item[(i)] 
$\Delta(4)=\Delta(1)=0$. 
\item[(ii)] 
$\Delta(4)\neq 0,\ \Delta(1)=\Delta(348)=\Delta(246)=0$. 
\item[(iii)] 
$\Delta(1)\neq 0,\ \Delta(4)=\Delta(128)=\Delta(136)=0$. 
\item[(iv)] 
$\Delta(4)\neq 0,\ \Delta(1)\neq 0,\ 
\Delta(348)=\Delta(136)=\Delta(128)=\Delta(246)=0$. 
\end{itemize}

The case (i) can not happen. Indeed, (i) implies that 
$q_6q_7q_8=q_2q_3=1$ and then we have $q_5=1$, which contradicts the 
assumption. 

(ii) implies that 
$$
q_1=q_3q_4q_8=q_2q_4q_6=q_2q_3q_5=q_1q_4q_6q_7q_8=q_5q_6q_7q_8=
q_1q_2q_3q_4=1. 
$$
From these relations, we obtain $q_3=q_6, q_4=q_5, q_2=q_8$ and $q_7=1$. 
These parameters form a component 
\begin{equation}
C_{(28|36|45)}=
\{(1,s_1,s_2,s_3,s_3,s_2,1,s_1)\in(\bC^*)^8\mid
s_1s_2s_3=1\}
\end{equation}
corresponding to the braid subarrangement $\{H_2,H_3,H_4,H_5,H_6,H_8\}$. 
Similarly, from (iii), we have 
\begin{equation}
C_{(15|26|38)}=
\{(s_1,s_2,s_3,1,s_1,s_2,1,s_3)\in(\bC^*)^8\mid
s_1s_2s_3=1\}
\end{equation}
corresponding to the braid subarrangement $\{H_1,H_2,H_3,H_5,H_6,H_8\}$. 

(iv) obviously implies that $D_1=D_2=D_4=D_5=0$. Since 
$q^{1/2}_{348}, 
q^{1/2}_{136},  
q^{1/2}_{128},  
q^{1/2}_{246}\in\{\pm 1\}$, other conditions are equivalent to 
\begin{equation*}
\begin{split}
D_3=0&\Longleftrightarrow q^{1/2}_{348}+q^{1/2}_{128}q^{1/2}_{1234}=0,\\
D_6=0&\Longleftrightarrow q^{1/2}_{136}+q^{1/2}_{246}q^{1/2}_{1234}=0. 
\end{split}
\end{equation*}
Since the choice of $q^{1/2}_i$ has freedom of the sign, we 
may assume that $q^{1/2}_{348}=q^{1/2}_{128}=1$. Then we have 
$q^{1/2}_{34}=q^{1/2}_{12}=q^{-1/2}_8$. Then $D_3=0$ is equivalent to 
$q^{1/2}_{1234}=-1$, which implies that 
\begin{equation*}
\begin{split}
D_3=0&\Longleftrightarrow q^{1/2}_{12}=q^{1/2}_{34}=q^{-1/2}_8=i\\
D_6=0&\Longleftrightarrow q^{1/2}_{13}=q^{1/2}_{24}(=\pm i). 
\end{split}
\end{equation*}
Set $q^{1/2}_1=\lambda$. Then 
$q^{1/2}_2=i\lambda^{-1}, 
q^{1/2}_3=\pm i\lambda^{-1}, 
q^{1/2}_4=\pm\lambda, 
q^{1/2}_5=\mp\lambda^2, 
q^{1/2}_6=\mp iq^{1/2}_{136}, 
q^{1/2}_7=\pm q^{1/2}_{136}\lambda^{-2}, 
q^{1/2}_8=-i$. The parameters $q_i=(q^{1/2}_i)^2$ form the following component 
(we set $s=\lambda^2$) 
\begin{equation*}
\Omega=
\{(s, -s^{-1}, -s^{-1}, s, s^2, -1, s^{-2}, -1)\in
(\bC^*)^8\mid s\in\bC^*
\}. 
\end{equation*}
It is the so-called translated component \cite{suc-trans, suc-fundam}. 

\subsection{$q_7\neq 1$}

This case is quite similar to the case $q_5\neq 1$. 
If $H^1(\scL_{\bm{q}})\neq 0$, then we have: 
\begin{equation}
q\in 
C_{5678}\cup C_{147}\cup
C_{(18|37|46)}\cup
C_{(16|27|48)}\cup\Omega. 
\end{equation}

\subsection{$q_5=q_7=1, q_8\neq 1$}

In this case, we compute $H^1$ with Figure \ref{fig:h8}. 

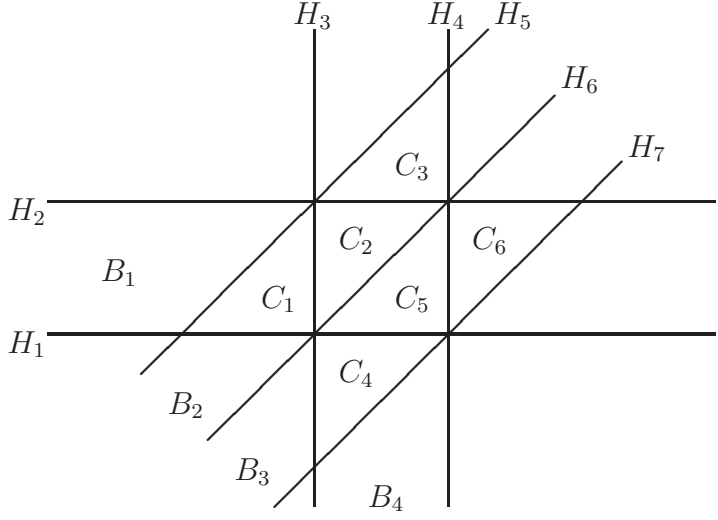
\begin{figure}[htbp]
\begin{picture}(350,200)(-25,0)
\thicklines

\multiput(150,10)(50,0){2}{\line(0,1){180}}
\multiput(50,75)(0,50){2}{\line(1,0){250}}

\multiput(135,10)(-25,25){3}{\line(1,1){130}}

\put(35,68){$H_1$}
\put(35,118){$H_2$}
\put(142,192){$H_3$}
\put(192,192){$H_4$}
\put(217,192){$H_5$}
\put(242,168){$H_6$}
\put(267,142){$H_7$}

\put(70,95){$B_1$}
\put(95,45){$B_2$}
\put(120,20){$B_3$}
\put(170,10){$B_4$}

\put(130,85){$C_1$}
\put(180,85){$C_5$}
\put(180,135){$C_3$}

\put(159,107){$C_2$}
\put(209,107){$C_6$}
\put(159,57){$C_4$}

\end{picture}
      \caption{The deleted $B_3$-arrangement with line at infinity $H_8$}
\label{fig:h8}
\end{figure}
The set of resonant bands $\RB_\scL$ is a subset of 
$\{B_1, B_2, B_3, B_4\}$. The coefficients of $\nabla(B_i)$ are 
as follows. 
\begin{equation*}
\begin{array}{c|cccccc}
&C_1&C_2&C_3&C_4&C_5&C_6\\
\hline
\nabla(B_1)&0&\Delta(3)&0&0&\Delta(36)&0\\
\nabla(B_2)&\Delta(1)&\Delta(13)&\Delta(123)&0&0&0\\
\nabla(B_3)&0&0&0&\Delta(3)&\Delta(13)&\Delta(134)\\
\nabla(B_4)&0&\Delta(16)&0&0&\Delta(1)&0
\end{array}
\end{equation*}
There are three multiple points $128, 5678$ and $348$ 
on $H_8$. We divide into $8$ cases. 
\begin{itemize}
\item[(i)] 
$q_{128}\neq 1, q_{5678}\neq 1, q_{348}\neq 1$. 
By Corollary \ref{cor:empty}, $H^1=0$. 
\item[(ii)] 
$q_{128}=1, q_{5678}\neq 1, q_{348}\neq 1$. 
By Theorem \ref{thm:1pt}, $H^1\neq 0$ if and only if 
$q_3=q_4=q_5=q_6=q_7=1$. Hence $\bm{q}=(q_1, \dots, q_8)$ is 
contained in $C_{128}$. 
\item[(iii)] 
$q_{128}\neq 1, q_{5678}= 1, q_{348}\neq 1$. 
By Theorem \ref{thm:1pt}, $H^1\neq 0$ if and only if 
$q_1=q_2=q_3=q_4=1$. Hence $\bm{q}=(q_1, \dots, q_8)$ is 
contained in $C_{5678}$. 
\item[(iv)] 
$q_{128}= 1, q_{5678}= 1, q_{348}\neq 1$. 
If $H^1(\scL_{\bm{q}})\neq 0$, then we can prove that 
$\bm{q}$ is contained in 
$$
\{(1,t^{-1},t,1,1,t^{-1},1,t)\mid t\in\bC^*\}\cup
\{(t^{-1},1,1,t,1,t^{-1},1,t)\mid t\in\bC^*\}, 
$$
which is a subset of $C_{(15|26|38)}\cup C_{(16|27|48)}$. 
\item[(v)] 
$q_{128}\neq 1, q_{5678}\neq 1, q_{348}= 1$. 
If $H^1(\scL_{\bm{q}})\neq 0$, then $\bm{q}\in C_{348}$. 
\item[(vi)] 
$q_{128}= 1, q_{5678}\neq 1, q_{348}= 1$. 
In this case, there are two resonant bands $B_1$ and $B_4$. 
Thus $H^1(\scL_{\bm{q}})\neq 0$ if and only if 
the minor 
$$
\det
\begin{pmatrix}
\Delta(3)&\Delta(36)\\
\Delta(16)&\Delta(1)
\end{pmatrix}
=-\Delta(6)\Delta(136)
$$
is equal to zero. Suppose that $\Delta(6)=0$. Then since 
$q_{34567}=q_5=q_7=1$, we have $q_{34}=1$, which implies $q_8=1$. 
Hence we may assume that $\Delta(6)\neq 0$ and $\Delta(136)=0$. We 
can conclude that $H^1\neq  0$ if and only if $\bm{q}$ is 
contained in 
\begin{equation*}
C_{(14|23|68)}=
\{(s_1,s_2,s_2,s_1,1,s_3,1,s_3)\mid s_i\in\bC^*, s_1s_2s_3=1\}. 
\end{equation*}

\item[(vii)] 
$q_{128}\neq 1, q_{5678}= 1, q_{348}= 1$. It is similar to the 
case (iv). We can prove that if $H^1\neq 0$, then $\bm{q}$ is 
contained in 
\begin{equation*}
C_{(28|36|45)}\cup
C_{(18|37|46)}. 
\end{equation*}

\item[(viii)] 
$q_{128}=1, q_{5678}= 1, q_{348}= 1$. 
We can prove that $\bm{q}\in C_{(14|23|68)}$. 
\end{itemize}

\subsection{$q_5=q_7=q_8=1, q_6\neq 1$}

We use Figure \ref{fig:h6}. This case is similar to the 
previous case. Two new components $C_{136}$ and $C_{246}$ appear. 
Other parameters are 
contained in components which have been already known.

\begin{figure}[htbp]
\begin{picture}(350,200)(-25,0)
\thicklines

\multiput(150,10)(50,0){2}{\line(0,1){180}}
\multiput(50,75)(0,50){2}{\line(1,0){250}}

\multiput(135,10)(-25,25){3}{\line(1,1){130}}

\put(35,68){$H_2$}
\put(35,118){$H_4$}
\put(142,192){$H_1$}
\put(192,192){$H_3$}
\put(217,192){$H_7$}
\put(242,168){$H_8$}
\put(267,142){$H_5$}

\put(70,95){$B_1$}
\put(95,45){$B_2$}
\put(120,20){$B_3$}
\put(170,10){$B_4$}

\put(130,85){$C_1$}
\put(180,85){$C_5$}
\put(180,135){$C_3$}

\put(159,107){$C_2$}
\put(209,107){$C_6$}
\put(159,57){$C_4$}

\end{picture}
      \caption{The deleted $B_3$-arrangement with line at infinity $H_6$}
\label{fig:h6}
\end{figure}
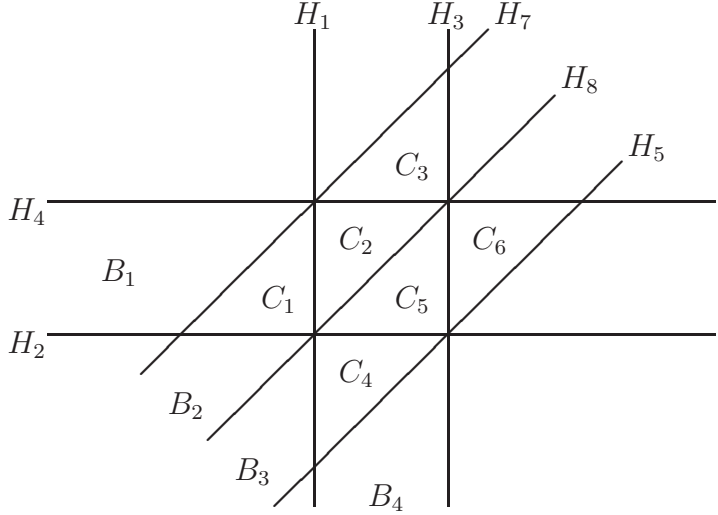

\subsection{$q_5=q_6=q_7=q_8=1, q_3\neq 1$}

We use Figure \ref{fig:h3}. There are three bands $B_1, B_2, B_3$. 

\begin{figure}[htbp]
\begin{picture}(350,200)(-25,0)
\thicklines

\multiput(150,10)(50,0){2}{\line(0,1){180}}
\multiput(50,75)(0,50){2}{\line(1,0){250}}

\put(135,10){\line(1,1){150}}
\put(100,25){\line(1,1){160}}
\put(250,25){\line(-1,1){160}}

\put(35,68){$H_8$}
\put(35,118){$H_4$}
\put(142,192){$H_1$}
\put(192,192){$H_6$}
\put(80,188){$H_7$}
\put(90,16){$H_2$}
\put(125,0){$H_5$}

\put(70,95){$B_2$}
\put(120,20){$B_3$}
\put(170,180){$B_1$}

\put(170,80){$C_4$}
\put(185,95){$C_3$}

\put(169,112){$C_2$}
\put(154,97){$C_1$}
\put(209,107){$C_5$}
\put(159,57){$C_6$}

\end{picture}
      \caption{The deleted $B_3$-arrangement with line at infinity $H_3$}
\label{fig:h3}
\end{figure}
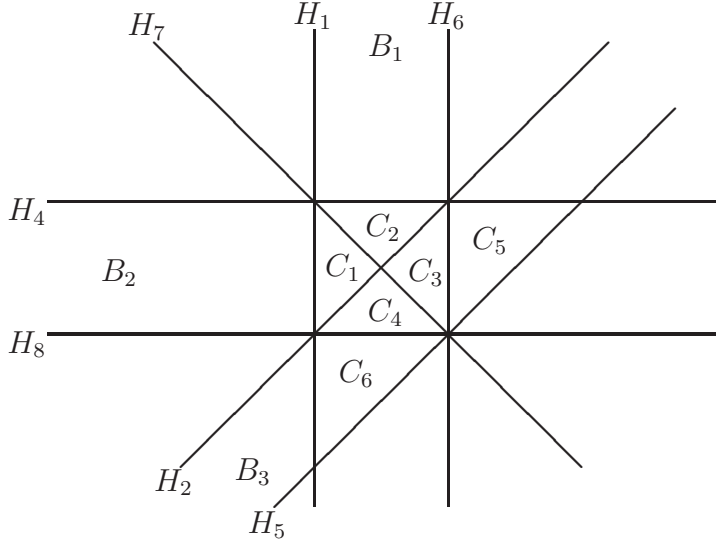

\begin{equation*}
\begin{array}{c|cccccc}
&C_1&C_2&C_3&C_4&C_5&C_6\\
\hline
\nabla(B_1)&\Delta(4)&\Delta(4)&\Delta(24)&\Delta(24)&0&\Delta(24)\\
\nabla(B_2)&\Delta(1)&\Delta(1)&\Delta(12)&\Delta(12)&\Delta(12)&0\\
\nabla(B_3)&0&0&\Delta(1)&\Delta(1)&\Delta(1)&\Delta(1)
\end{array}
\end{equation*}
If bands $B_1$, $B_2$ and $B_3$ are all resonant, then 
$\Delta(13)=\Delta(24)=\Delta(12)=\Delta(34)=
\Delta(14)=\Delta(23)=0$. We have 
$q_1=q_2=q_3=q_4=-1$, and so $\bm{q}\in C_{(14|23|68)}$. Other 
cases are similar. Consequently, in this case, $H^1\neq 0$ 
implies that 
\begin{equation*}
\bm{q}\in
C_{235}\cup C_{348}\cup C_{(14|23|68)}. 
\end{equation*}
So the parameter $\bm{q}$ is contained in the known components. 

\subsection{$q_3=q_5=q_6=q_7=q_8=1$}

The remaining case is easily seen that 
$H^1(\scL_{\bm{q}})= 0$ if $\bm{q}\neq (1,1,\dots, 1)$. 

\medskip

Consequently, we have the decomposition of the 
characteristic variety as follows. 
\begin{equation*}
\begin{split}
V_1(M(\A))=&
C_{136}\cup
C_{147}\cup
C_{235}\cup
C_{128}\cup
C_{246}\cup
C_{348}\cup
C_{5678}\cup
\\
&
C_{(14|23|68)}\cup
C_{(28|36|45)}\cup
C_{(15|26|38)}\cup
C_{(18|37|46)}\cup
C_{(16|27|48)}\cup
\Omega. 
\end{split}
\end{equation*}

\medskip

\noindent
{\bf Acknowledgement.} 
Part of this work was done while the author 
was visiting the Universidad de Zaragoza. 
The author thanks 
Professor E. Artal Bartolo and Professor J. I. Cogolludo-Agust\'in 
for their support, hospitality and valuable discussions. 
This work is supported by JSPS Grant-in-Aid for Young Scientists (B). 


\end{document}